\documentclass[12pt, a4paper, leqno]{amsart}
\usepackage[utf8]{inputenc}
\usepackage{amsmath}
\usepackage{amsfonts}
\usepackage{amssymb}
\usepackage{times}
\usepackage[T1]{fontenc} 
\usepackage{url} 
\usepackage{color,esint}
\usepackage[dvipsnames]{xcolor}
\usepackage[colorlinks, allcolors=RedViolet,pdfstartview=,pdfpagemode=UseNone]{hyperref} 
\usepackage{graphicx} 
\usepackage{enumitem}
\usepackage{tabularx}
    \usepackage{mathtools}
\usepackage{pgf,tikz}
\usepackage{mathrsfs}
\usetikzlibrary{arrows}

\let\oldmarginpar\marginpar
\renewcommand\marginpar[1]{\-\oldmarginpar[\raggedleft\footnotesize #1]%
{\raggedright\footnotesize #1}}

\usepackage{amsthm}
\theoremstyle{plain}
\newtheorem{thm}[equation]{Theorem}
\newtheorem{lem}[equation]{Lemma}
\newtheorem{prop}[equation]{Proposition}
\newtheorem{cor}[equation]{Corollary}

\theoremstyle{definition}
\newtheorem{defn}[equation]{Definition}

\theoremstyle{remark}
\newtheorem{rem}[equation]{Remark}

\numberwithin{equation}{section}

\newcommand{\R}{\mathbb{R}}

\newcommand{\Rn}{\mathbb{R}^n}

\renewcommand{\phi}{\varphi}
\def\le{\leqslant}
\def\leq{\leqslant}
\def\ge{\geqslant}
\def\geq{\geqslant}
\def\phi{\varphi}
\def\rho{\varrho}
\def\vartheta{\theta}

\newcommand{\Phiw}{\Phi_{\text{\rm w}}}
\newcommand{\Phic}{\Phi_{\text{\rm c}}}

\def\esssup{\operatornamewithlimits{ess\,sup}}

\def\essinf{\operatornamewithlimits{ess\,inf}}

\DeclareMathOperator{\divop}{div}
\renewcommand{\div}{\divop}

\def\loc{{\rm loc}}

\newcommand{\inc}[1]{\hyperref[def:aInc]{{\normalfont(Inc){\ensuremath{_{#1}}}}}}
\newcommand{\dec}[1]{\hyperref[def:aDec]{{\normalfont(Dec){\ensuremath{_{#1}}}}}}
\newcommand{\ainc}[1]{\hyperref[def:aInc]{{\normalfont(aInc){\ensuremath{_{#1}}}}}}
\newcommand{\adec}[1]{\hyperref[def:aDec]{{\normalfont(aDec){\ensuremath{_{#1}}}}}}
\newcommand{\adeci}[1]{\hyperref[def:aDeci]{{\normalfont(aDec){\ensuremath{_{#1}^\infty}}}}}
\newcommand{\azero}{\hyperref[def:a0]{{\normalfont(A0)}}}
\newcommand{\aone}{\hyperref[def:a1]{{\normalfont(A1)}}}
\newcommand{\aones}[1]{\hyperref[def:a1s]{{\normalfont(A1-{\ensuremath{{#1}})}}}}

\date{\today}

\begin{document}

\title
[Unbounded supersolutions with generalized Orlicz growth]
{The weak Harnack inequality for unbounded supersolutions of equations with generalized Orlicz growth}
\author[Benyaiche, Harjulehto, Hästö, and Karppinen]{Allami Benyaiche, Petteri Harjulehto, Peter Hästö, and Arttu Karppinen}


\begin{abstract}
We study unbounded weak supersolutions of elliptic partial differential equations with generalized Orlicz 
(Musielak--Orlicz) growth. 
We show that they satisfy the weak Harnack inequality with optimal exponent provided that they belong to 
a suitable Lebesgue or Sobolev space. 
Furthermore, we establish the sharpness of our central assumptions. 
\end{abstract}

\maketitle


\section{Introduction}

We prove the weak Harnack inequality for unbounded supersolutions of partial differential equations 
with generalized Orlicz growth (also known as Musielak--Orlicz growth). 
A general principle states that only intrinsic Harnack inequalities 
are possible for PDEs which are not scaling invariant, and we do indeed find that the constant 
in the weak Harnack inequality depends on the norm of the supersolution. 
With dependence on the $L^s$- or $W^{1,s}$-norms, our result requires that 
$s$ be sufficiently large, namely, $s\ge \max\{s_1,s_2\}$, where $s_1$ depends continuity of the 
generalized Orlicz functional 
and is shown to be sharp and $s_2$ depends on the global growth of the functional and 
does not occur in any previously known special cases.  
The result is new even for the special case of double phase functionals \cite{BarCM18}.
Our framework includes also the following special cases in which the weak Harnack inequality 
was not known before, even for bounded solutions: 
perturbed variable exponent \cite{GP13,LiaSZ18,Ok0,Ok18}, 
Orlicz variable exponent \cite{CapCF18, GPRT17}, 
degenerate double phase \cite{BCM2,BOh2}, 
Orlicz double phase \cite{BOh3}, 
variable exponent double phase \cite{MizNOS20,RagT20}, 
triple phase \cite{DeFO19}, and
double variable exponent \cite{CenRR18, RadRSZ_pp,ZhaR18}. 

Let us introduce the context of this paper. 
Minimizers and (weak) solutions of
\begin{align*}
\inf \int_\Omega F(x, \nabla u) \, dx \quad \text{and}\quad 
-\div(f(x,\nabla u)) = 0
\end{align*}
have been actively studied during recent years when $F$ or $f$ has generalized Orlicz growth. For instance, 
solutions with given boudary values exist \cite{ChlGZ19,GwiSZ18, HarHK16}, 
minimizers or solutions with given boundary values are locally bounded, satisfy Harnack's inequality and belong to $C^{0, \alpha}_{\loc}$ \cite{BenK_pp, HarHT17, WanLZ19} or $C^{1, \alpha}_{\loc}$ \cite{HasO_pp}, 
quasiminimizers satisfy a reverse H\"older inequality \cite{HarHK18}, 
$\omega$-minimizers are locally Hölder continuous \cite{HarHL_pp}, 
minimizers for the obstacle problem are continuous \cite{Kar19} and the boundary Harnack inequality 
holds for harmonic functions \cite{ChlZ_pp}. 
In most cases the assumptions in these results coincides with optimal assumptions in well-known special cases. 
The important special cases are the variable exponent case $F(x,\xi) := |\xi|^{p(x)}$ 
\cite{Alk97, HarHL09, HarKL07, Jul15, Toi12}, the Orlicz case $F(x,\xi) := \phi(|\xi|)$ 
\cite{ArrH18, Lie87, Lie91}, 
and  the double phase case $F(x,\xi) := |\xi|^{p} + a(x) |\xi|^q$ \cite{BarCM15, ColM15a, ColM15b}.
%
The surveys \cite{Chl18, Mar20} include more references of variational problems and partial differential equations of generalized Orlicz growth, while the recent monographs \cite{HarH19, LanM19} present the theory of the 
underlying function spaces. 

In \cite{BenK_pp}, the weak Harnack inequality for bounded supersolutions and Harnack's inequality for solutions were proved by Moser's iteration in the generalized Orlicz case. 
In those results the constants depend on $L^\infty_\loc$-norm of the function. 
In \cite{BarCM15}, Harnack's inequality has been proved in the double phase case, 
and the constant depends on $L^\infty_\loc$-norm of the function. 

We want to study unbounded supersolutions, so we need the constants not to depend on the $L^\infty$-norm. 
In the variable exponent case, the constant in the weak Harnack inequality depends 
on $L^t_\loc$-norm of the function and $t>0$ can be chosen arbitrarily small; furthermore, 
an example shows that the constant in  Harnack's inequality cannot be independent of the function, in contrast to the constant exponent case \cite{HarKL07}. We extend these results to the generalized Orlicz case with sharp assumptions on the continuity of $\phi$. 

We assume that $f:\Omega\times  \Rn \to \Rn$ satisfies the following $\phi$-growth conditions:
\begin{equation}\label{eq:growth} 
\nu \phi(x,|\xi|) \le f(x,\xi)\cdot \xi
\quad\text{and}\quad
|f(x,\xi)|\, |\xi| \le \Lambda \phi(x,|\xi|)
\end{equation}
for all $x\in \Omega$ and $\xi\in\Rn$, and fixed but arbitrary constants $0<\nu \leq \Lambda$.
We are interested in local  (weak) supersolutions:

\begin{defn}
A function $u \in W^{1, \phi}_\loc(\Omega)$ is a \emph{supersolution} if 
\begin{align}
\label{eq:A-eq}
\int_{\Omega} f(x,\nabla u) \cdot \nabla h \, dx \ge 0,
\end{align}
for all non-negative $h \in W^{1, \phi}(\Omega)$ with compact support in $\Omega$.
\end{defn}

We define the limiting exponent 
\[
\ell(p):= \frac{p^*}{p'} = 
\begin{cases}
\frac n{n-p}(p-1) & \quad\text{if }p<n,\\
\infty & \quad\text{if }p\ge n;
\end{cases}
\]
the ratio of the Sobolev exponent $p^*$ and the H\"older exponent $p'$. 
Since $u_p(x):=|x|^{-\frac{n-p}{p-1}}$ is a supersolution of the $p$-Laplace equation 
$-\div(|\nabla u|^{p-2}\nabla u) = 0$ in the unit ball and since $u_p\not\in L^{\ell(p)}(B_1)$, 
we see that $\ell(p)$ is an upper bound on the exponent of integrability of $p$-supersolutions. 

The following is a special case of our main result, Theorem~\ref{thm:weakHarnackGeneral}
together with Propositions \ref{prop:vBound} and \ref{prop:betaBound}. 
Theorem~\ref{thm:weakHarnackGeneral} contains also 
the cases $\|u\|_{L^\omega(B_{2R})}\le d$ and $\|u\|_{W^{1,\omega}(B_{2R})}\le d$ for general 
$\omega\in \Phiw(\Omega)$.
The last statement in the next theorem follows from the example in Section~\ref{sect:example}.
Note that with the choice $s=\infty$ we recover as a special case previously known results for bounded 
solutions \cite{BenK_pp} with the correct, \aones{n} assumption.

\begin{thm}[The weak Harnack inequality]\label{thm:weakHarnack}
Suppose $\phi$ satisfies \azero{}, \ainc{p} and \adec{q}, $1<p\le q<\infty$. 
Let $u$ be a non-negative supersolution to \eqref{eq:A-eq} in $B_{2R}$.
Assume that one of the following holds:
\begin{enumerate}
\item
$\phi$ satisfies \aones{s_*} and $\|u\|_{L^s(B_{2R})}\le d$, where $s_* := \frac{ns}{n+s}$ and $s\in [q-p,\infty]$.
\item
$\phi$ satisfies \aone{} and $\|u\|_{W^{1,\phi}(B_{2R})}\le d$.
\end{enumerate}
Then there exist  positive constants $\ell_0$ and $C$ 
such that the weak Harnack inequality holds: 
\begin{align*}
\bigg (\fint_{B_{2R}} (u +R)^{\ell_0} \, dx \bigg)^{\frac{1}{\ell_0}}  \le C (\essinf_{B_R} u + R).
\end{align*}
If (1) holds with $s>\max\{\frac np,1\}(q-p)$ or if (2) holds with $p^*>q$, then the weak Harnack inequality 
holds for any $\ell_0<\ell(p)$. 

The \aones{s_*} assumption is sharp, since 
for any $s'<s_*$ if, instead of (1), $\phi$ satisfies \aones{s'} and $\|u\|_{L^s(B_{2R})}\le d$, 
then the weak Harnack inequality need not hold. 
\end{thm}

The next result follows by Corollary~\ref{cor:doublePhase-2} and the counter-example given in 
Section~\ref{sect:example}. Corresponding corollaries could be formulated in the other cases of 
double phase type functionals listed in the beginning of this section. 

\begin{cor}\label{cor:doublePhase}
Let $\phi(x,t):=t^p + a(x)t^q$ be the double phase functional with 
$a^\lambda\in C^{0,\alpha\lambda}(\Omega)$ for some $\lambda>0$. 
Let $u$ be a non-negative supersolution to \eqref{eq:A-eq} in $B_{2R}$. 
If $u\in L^s(\Omega)$ with  
\[
\alpha \ge (\tfrac{n}s+1)(q-p) 
\quad\text{and}\quad
s\in [q-p,\infty],
\]
then there exist  positive constants $\ell_0$ and $C$ 
such that
\begin{align*}
\bigg (\fint_{B_{2R}} (u +R)^{\ell_0} \, dx \bigg)^{\frac{1}{\ell_0}}  \le C (\essinf_{B_R} u + R).
\end{align*}
If additionally $s>\max\{\frac np,1\}(q-p)$, then the inequality holds for any $\ell_0<\ell(p)$. 

The bound on $\alpha$ is sharp, since 
for every $\alpha < (\tfrac{n}s+1)(q-p)$, there exists $a\in C^{0,\alpha}(\Omega)$ for which 
the weak Harnack inequality does not hold.  
\end{cor}

As the example in Section~\ref{sect:example} shows, the assumption $\alpha \ge (\tfrac{n}s+1)(q-p)$ is sharp 
and this restriction has been previously encountered in \cite{Ok20}, see also \cite{BarCM18}. 
The assumption $s\ge q-p$ is a consequence of considering supersolutions as it can be omitted 
if $u$ is a solution. It is especially interesting to note that such a restriction does not 
occur in the variable exponent case and it is another example that the double phase functional is 
more subtle than the variable exponent case \cite{ColM15a}. 

\begin{rem}\label{rem:plusR}
Compared to the classical Harnack inequality, our estimate contains an extra $+R$-term. It is not know 
whether this is necessary, 
but the same phenomenon occurs when $\phi(x,t)= t^{p(x)}$ \cite{Alk97, HarKL07, HarKLMP08}, unless the 
exponent $p$ is assumed to belong to $C^1$ and slightly different Harnack's inequality is used \cite{Jul15}. 
In the Orlicz and the double phase cases the extra $+R$ term is not needed \cite{ArrH18, BarCM15, Lie91}.
\end{rem}

\begin{rem}\label{rem:japaneseCase}
In the case of the basic double phase functional, namely, when $\lambda=1$ in Corollary~\ref{cor:doublePhase}, 
the assumption $\alpha \ge (\frac{n}s+1)(q-p)$ implies $s\ge n(q-p) \ge q-p$ and $s>\frac np(q-p)$, 
since $\alpha\le 1$ and $q\ge 1$. 
Mizuta, Ohno and Shimomura \cite{MizOS20} (see also \cite{HarH_pp}) have considered 
the functional $\phi(x,t):=t^p + (a(x)t)^q$, 
which corresponds to the case $\lambda=\frac 1q$ above. Also in this case, the first assumption 
implies the latter two. However, if $\lambda<\frac1q$, then the first condition can hold while 
the others do not. 

Note that the parameter $\lambda$ does not impact the restrictions in 
Corollary~\ref{cor:doublePhase}. The reason for this is that only the growth of $a$ 
at $a=0$ is important, see \cite[Proposition~7.2.2]{HarH19}.
\end{rem}


\section{Preliminaries}

We briefly introduce our assumptions. More information about $L^\phi$-spaces can be found in \cite{HarH19}. 
\textit{Almost increasing} means that a function satisfies $f(s) \le L f(t)$ for all $s<t$ and some constant $L\ge 1$. 
If there exists a constant $C$ such that $f(x) \leq C g(x)$ for almost every $x$, then we write $f \lesssim g$. If $f\lesssim g\lesssim f$, then we write $f \approx g$.

\begin{defn}
We say that $\phi: \Omega\times [0, \infty) \to [0, \infty]$ is a 
\textit{weak $\Phi$-function}, and write $\phi \in \Phiw(\Omega)$, if 
the following conditions hold:
\begin{itemize}
\item 
For every measurable function $f:\Omega\to \R$ the function $x \mapsto \phi(x, f(x))$ is measurable 
and for every $x \in \Omega$ the function $t \mapsto \phi(x, t)$ is non-decreasing.
\item 
$\displaystyle \phi(x, 0) = \lim_{t \to 0^+} \phi(x,t) =0$  and $\displaystyle \lim_{t \to \infty}\phi(x,t)=\infty$ for every $x\in \Omega$.
\item 
The function $t \mapsto \frac{\phi(x, t)}t$ is $L$-almost increasing on $(0,\infty)$ with $L$ independent of $x$.
\end{itemize}
If $\phi\in\Phiw(\Omega)$ is additionally convex and left-continuous, then $\phi$ is a 
\textit{convex $\Phi$-function}, and we write $\phi \in \Phic(\Omega)$. If $\phi$ does not depend on $x$, then we omit the set  and write $\phi \in \Phiw$ or $\phi \in \Phic$.
\end{defn}



We denote $\phi^+_{A}(t) = \esssup_{x \in A\cap \Omega} \phi(x,t)$ and 
$\phi^-_{A}(t) = \essinf_{x \in A\cap \Omega} \phi(x,t)$. 
We define several conditions. See Table~\ref{tab:examples} for an intuition of their meaning 
in special cases. 
Let $p,q,s>0$ and let $\omega:\Omega\times [0,\infty)\to [0,\infty)$ be almost increasing. 
We say that $\phi:\Omega\times [0,\infty)\to [0,\infty)$ satisfies 
\begin{itemize}
\item[(A0)]\label{def:a0}
if there exists $\beta \in(0, 1]$ such that $ \beta \le \phi^{-1}(x,
1) \le \frac1{\beta}$ for a.e.\ $x \in \Omega$, 

\item[(A1-$\omega$)] \label{def:a1s}
if there exists $\beta \in (0,1]$ such that,  for every ball $B$ and a.e.\ $x,y\in B \cap \Omega$,
\[ 
\phi(x,\beta t) \le \phi(y,t) \quad\text{when}\quad  \omega_B^-(t) \in \bigg[1, \frac{1}{|B|}\bigg];
\]

\item[(A1-$s$)]
if it satisfies \aones{\omega} for $\omega(x,t):=t^s$;
 
\item[(A1)]\label{def:a1}
if it satisfies \aones{\phi};

\item[(aInc)$_p$] \label{def:aInc} if
$t \mapsto \frac{\phi(x,t)}{t^{p}}$ is $L_p$-almost 
increasing in $(0,\infty)$ for some $L_p\ge 1$ and a.e.\ $x\in\Omega$;

\item[(aDec)$_q$] \label{def:aDec}
if
$t \mapsto \frac{\phi(x,t)}{t^{q}}$ is $L_q$-almost 
decreasing in $(0,\infty)$ for some $L_q\ge 1$ and a.e.\ $x\in\Omega$.
\end{itemize} 
We say that \ainc{} holds if \ainc{p} holds for some $p>1$, and similarly for \adec{}.
If in the definition of \ainc{p} we have $L_p=1$, then we say that $\phi$ satisfies \inc{p}, 
similarly for \dec{q}.

\begin{table}[h]
\centerline{
\begin{tabular}{l|ccccc}
$\phi(x,t):=$ & \azero{} & \aone{}& \aones{s}& \ainc{} & \adec{} \\
\hline
$\phi(t)$ & $\text{true}$ & $\text{true}$ & $\text{true}$ & $\nabla_2$ & $\Delta_2$ \\
$t^{p(x)}a(x)$ & $a\approx 1$ & $p\in C^{\log}$ & $p\in C^{\log}$ & $p^->1$ & $p^+<\infty$ \\
$t^{p(x)}\log(e+t)$ & $\text{true}$ & $p\in C^{\log}$ & $p\in C^{\log}$ & $p^->1$ & $p^+<\infty$ \\
$t^p + a(x) t^q$ & $a\in L^\infty$ & $a\in C^{0, \frac np (q-p)}$ & $a\in C^{0, \frac ns(q-p)}$ & $p>1$ & $q<\infty$ \\
\end{tabular}
}
\caption{Assumptions in special cases}\label{tab:examples}
\end{table}

We note that \aones{\omega} is a new condition introduced in this article 
to combine \aone{} and \aones{n} as well as other cases. Basically, \aones{\omega} is the 
appropriate assumption if we have \textit{a priori} information that the 
solution is in space $W^{1,\omega}$ or the corresponding Lebesgue or H\"older space. 
The most important cases are $\omega=\phi$ and $\omega(x,t)=t^s$, but we may, for instance, 
consider $\omega=\phi^{1+\epsilon}$ if we have some higher integrability information.

By \cite[Section~4.1]{HarH19}, \azero{} can be stated equivalently as the existence of $\beta \in(0,1]$ such that 
$\phi(x,\beta) \le 1\le \phi(x,1/\beta)$ for almost every $x \in \Omega$. If $\phi$ satisfies \azero{}, then 
\aone{} is equivalent to the condition that  there exists $\beta\in (0,1)$ such that
\[
\beta \phi^{-1}(x, t) \le \phi^{-1} (y, t) 
\quad\text{when}\quad 
t \in \bigg[1, \frac{1}{|B|}\bigg]
\]
for every ball $B$ and a.e.\ $x,y\in B \cap \Omega$ \cite[Section~4.2]{HarH19}. 

\begin{rem}
\label{rem:incdec}
Assume that the derivative $\phi'$ with respect to the second variable exists. Then 
\[
\frac d{dt}\frac{\phi(x,t)}{t^p} = \frac{\phi'(x, t) t^p- p t^{p-1}\phi(x,t)}{t^{2p}}.
\]
If $\phi$ satisfies \inc{p}, then the derivative is non-negative and so
$\frac{t\phi'(x,t)}{\phi(x,t)} \ge p$.
Similarly, $\frac{t\phi'(x,t)}{\phi(x,t)} \le q$ if $\phi$ satisfies \dec{q}.
If, on the other hand, 
\[
p\le \frac{t\phi'(x,t)}{\phi(x,t)} \le q
\] 
then $\frac d{dt}\frac{\phi(x,t)}{t^p}$ is non-negative and 
$\frac d{dt}\frac{\phi(x,t)}{t^q}$ is non-positive and hence \inc{p} and \dec{q} hold. 
Moreover,  the double inequality implies also that $\phi'$ satisfies \ainc{p-1} and \adec{q-1} 
if $\phi$ satisfies \inc{p} and \dec{q}.
\end{rem}


\begin{defn}\label{def:Lphi}
Let $\phi \in \Phiw(\Omega)$ and define the \textit{modular} 
$\varrho_\phi$ for $u\in L^0(\Omega)$, the set of measurable functions in $\Omega$, by 
  \begin{align*}
    \varrho_\phi(u) &:= \int_\Omega \phi(x, |u(x)|)\,dx.
  \end{align*}
The \emph{generalized Orlicz space}, also called Musielak--Orlicz space, is defined as the set 
  \begin{align*}
    L^\phi(\Omega) &:= \big\{u \in L^0(\Omega) \colon
      \lim_{\lambda \to 0^+} \varrho_\phi(\lambda u) = 0\big\}
  \end{align*}
equipped with the (Luxemburg) norm 
  \begin{align*}
    \|u\|_{L^\phi(\Omega)} &:= \inf \Big\{ \lambda>0 \colon \varrho_\phi\Big(
      \frac{u}{\lambda}  \Big) \leq 1\Big\}.
  \end{align*}
If the set is clear from context we abbreviate $\|u\|_{L^\phi(\Omega)}$ by $\|u\|_{\phi}$.
\end{defn}

We denote by $\phi^*$ the conjugate $\Phi$-function, defined by 
\[
\phi^*(x,t):= \sup_{s\ge 0} (st-\phi(x,s)). 
\]
By this definition, we have Young's inequality $st \le \phi(x,s) + \phi^*(x,t)$. 
H\"older's inequality holds in generalized Orlicz spaces for $\phi\in\Phiw(\Omega)$ with constant $2$
\cite[Lemma~3.2.13]{HarH19}:
\[
\int_\Omega |u|\, |v|\, dx \le 2 \|u\|_\phi \|v\|_{\phi^*(\cdot)}.
\]

\begin{defn}
A function $u \in L^\phi(\Omega)$ belongs to the
\textit{Orlicz--Sobolev space $W^{1, \phi}(\Omega)$} if its weak partial derivatives $\partial_1 u, \ldots, \partial_n u$ exist and belong to the space $L^{\phi}(\Omega)$.
For $u \in W^{1,\phi}(\Omega)$, we define the norm
\[
\| u \|_{W^{1,\phi}(\Omega)} := \| u \|_{L^\phi(\Omega)} + \| \nabla u \|_{L^\phi(\Omega)}.
\]
Here $\| \nabla u \|_{\phi}$ is a abbreviation of $\big{\|} | \nabla u | \big{\|}_{\phi}$.
Again, if $\Omega$ is clear from context, we abbreviate $\| u \|_{W^{1,\phi}(\Omega)}$ by $\| u \|_{1,\phi}$.
\end{defn}



We conclude the section by proving an appropriate version of the Caccioppoli inequality.
We denote by $\eta$ a cut-off function in $B_R$, more precisely, 
$\eta \in C_0^{\infty}(B_R)$, $\chi_{B_{\sigma R}} \leq \eta \leq \chi_{B_R}$ 
and $|\nabla \eta| \leq \frac{2}{(1-\sigma)R}$, where $\sigma \in (0, 1)$. Note that the auxiliary function 
$\psi$ is independent of $x$ in the next lemma. This avoids assumptions 
regarding the differentiability of $\psi$ in the space variable, but it does mean that 
the application of the lemma later on is more complicated compared to classical, standard growth 
cases where one simply choses $\psi=\phi$. 

\begin{lem}[Caccioppoli inequality]
\label{lem:caccioppoli}
Suppose $\phi \in \Phiw(\Omega)$ satisfies \azero{} and \adec{q}, 
and let $\psi\in \Phiw$ be differentiable and satisfy \azero{}, 
\inc{p_1} and \dec{q_1}, $p_1, q_1\ge 1$. Let $u$ be a non-negative supersolution 
to equation \eqref{eq:A-eq} and $\eta$ be a cut-off function in $B_R\subset \Omega$. 
For any $\ell > \frac1{p_1}$ and $s\ge q$, 
\begin{align*}
\int_{B_R} \phi(x, |\nabla u|) \psi( \tfrac{u+R}{R} )^{-\ell} \eta^s \, dx 
\leq 
\big(\tfrac{c(L_q) s\Lambda}{(1-\sigma)(p_1\ell-1)\nu}\big)^q \int_{B_R} \psi (\tfrac{u+R}{R} )^{-\ell} \phi(x,\tfrac{u+R}{R} ) \eta^{s-q} \, dx.
\end{align*}
\end{lem}

\begin{proof}
Let us simplify the notation by denoting $\tilde u := u + R$ and $v := \frac{\tilde u}{R}$. 
Since $\nabla u = \nabla \tilde u$, we see that $\tilde u$ is still a supersolution. 
Since $v\ge 1$, $\psi(v)^{-\ell} \eta^{s}\le c_1$ by \azero{} and \inc{p_1}. 

We would like to test equation \eqref{eq:A-eq} with $h := \psi(v)^{-\ell} \eta^{s} \tilde u$. Let us first check that $h$ is a valid test function, that is $h \in W^{1,\phi}(B_R)$ and has compact support in $\Omega$. As $\tilde u \in L^{\phi}(B_R)$ and $|h|\le c_1 \tilde u$ it is immediate that $h \in L^{\phi}(B_R)$. By a direct calculation we have
\begin{align*}
\begin{split}
\nabla h 
&= -\ell \psi(v)^{-\ell-1} \eta^{s} \tilde u \psi'(v) \nabla v + 
s \psi(v)^{-\ell} \eta^{s-1} \tilde u \nabla \eta + \psi(v)^{-\ell} \eta^s \nabla \tilde u. 
\end{split}
\end{align*}
Note that $\tilde u \nabla v =  v \nabla \tilde u$. We use Remark~\ref{rem:incdec} and get
\begin{align*}
\big|\ell \psi(v)^{-\ell-1} \eta^{s}  \psi'(v) v \nabla \tilde u \big|
\leq 
q_1\ell c_1 |\nabla \tilde u| \in L^{\phi}(B_R).
\end{align*}
For the third term, we obtain $|\psi(v)^{-\ell} \eta^{s} \nabla \tilde u |
\leq  c_1 |\nabla \tilde u| \in L^{\phi}(B_R)$.
The term with $\nabla \eta$ is treated as $h$ itself. Thus $h \in W^{1, \phi}(B_R)$.
Since $s>0$ and $\eta\in C^\infty_0(B_R)$, $h$ has compact support in 
$\Omega$ and so it is a valid test-function. 

We next calculate
\begin{align*}
f(x,\nabla \tilde u) \cdot \nabla h 
&=\psi(v)^{-\ell-1} \eta^{s} [-\ell \psi'(v) v + \psi(v)] f(x,\nabla \tilde u) \cdot \nabla \tilde u\\ 
&\qquad+ s\psi(v)^{-\ell} \eta^{s-1} \tilde u \, f(x,\nabla \tilde u) \cdot \nabla \eta .
\end{align*}
Since $\tilde u$ is a supersolution, we have
$\int_{B_R} f(x, \nabla \tilde u) \cdot \nabla h \, dx \ge 0$,
which implies with the growth conditions \eqref{eq:growth} that 
\begin{align*}
\nu \int_{B_R} \phi(x, |\nabla \tilde u|) \psi(v)^{-\ell-1} \eta^{s}
[\ell  \psi'(v) v - \psi(v)]
\, dx \le
s \Lambda \int_{B_R} \tfrac{\phi(x, |\nabla \tilde u|)}{|\nabla \tilde u|} |\nabla \eta| \, \eta^{s-1}\psi(v)^{-\ell}\tilde u \, dx.
\end{align*}
By Remark~\ref{rem:incdec}, $p_1 \psi(t) \le \psi'(t) t$ so we conclude that 
\begin{align*}
& [p_1\ell- 1]\nu \int_{B_R} \phi(x, |\nabla \tilde u|) \psi(v)^{-\ell} \eta^{s}\, dx 
\le
s\Lambda \int_{B_R} \tfrac{\phi(x, |\nabla \tilde u|)}{|\nabla \tilde u|}|\nabla \eta |\, \eta^{s-1}\psi(v)^{-\ell}\tilde u
\, dx.
\end{align*}
Here $p_1\ell- 1$ is positive, since $\ell> \frac1{p_1}$.

Recalling that $|\nabla \eta| \tilde u \leq \frac{2}{(1-\sigma)R}\tilde u
=\frac{2}{1-\sigma} v$, we arrive at
\begin{align*}
\int_{B_R} \phi(x, |\nabla \tilde u|)  \psi(v)^{-\ell} \eta^{s} \, dx
\leq 
\frac{C}{1-\sigma} \int_{B_R} \tfrac{\phi(x, |\nabla \tilde u|)}{|\nabla \tilde u|} v \psi(v)^{-\ell} \eta^{s-1}\, dx.
\end{align*}
By Young's inequality \cite[(2.4.2)]{HarH19}
\begin{equation*}
\tfrac{\phi(x, |\nabla \tilde u|)}{|\nabla \tilde u|} v 
\le 
\phi\big(x, \epsilon^{-\frac1{q'}} v\big)
+  \phi^*\big(x, \epsilon^{\frac1{q'}} \tfrac{\phi(x, |\nabla \tilde u|)}{|\nabla \tilde u|}\big). 
\end{equation*}
For the first term on the right hand side we use \adec{q} of $\phi$.
For the second term we first  use
\ainc{q'} of $\phi^*$ \cite[Proposition 2.4.9]{HarH19} and then
$\phi^*(\frac{\phi(t)}t) \le \phi(t)$ (comment after \cite[Theorem 2.4.10]{HarH19}), and obtain
\begin{equation*}
\tfrac{\phi(x, |\nabla \tilde u|)}{|\nabla \tilde u|} v 
\lesssim
\epsilon^{1-q} \phi(x,v) + \epsilon \phi(x,|\nabla u|) .
\end{equation*}

Finally, we choose $\epsilon:=\frac{1-\sigma}{2C}\eta(x)$ and so 
\begin{align*}
&\int_{B_R} \phi(x, |\nabla \tilde u|)  \psi(v)^{-\ell} \eta^{s} \, dx\\
 &\quad\leq \frac{1}{2} \int_{B_R} \phi(x, |\nabla \tilde u|) \psi(v)^{-\ell} \eta^{s} \, dx 
 + \frac{C}{(1-\sigma)^q} \int_{B_R} \psi(v)^{-\ell} \phi(x,v) \eta^{s-q} \, dx.
\end{align*}
The first term on the right-hand side can be absorbed in the left-hand side. This gives the claim. 
\end{proof}


\section{The weak Harnack inequality}

The plan of the proof follows the usual scheme of Moser iteration. 
We first show that the infimum is bounded below by an integral-mean with negative power. 
We then prove a reverse H\"older-type inequality for positive exponents below a certain 
threshold. These steps are proved by iteration. Jumping over zero is the final piece and it is 
accomplished by the John--Nirenberg lemma. 

We next define a 
differentiable approximation of $\phi_{B_R}^-$ with nice growth properties. 
We assume that $\phi\in \Phiw(\Omega)$ satisfies \adec{} and \ainc{p}, $p\ge 1$.

\begin{defn}\label{def:psi}
We define an auxiliary weak $\Phi$-function $\psi_r$ by setting $\psi_r(0):=0$ and 
\[
\psi_r(t):= \int_0^t \tau^{p-1}\sup_{s\in (0, \tau]} \frac{\phi_{B_r}^-(s)}{s^p} \, d\tau
\quad\text{for }t>0.
\] 
\end{defn}

Then $\frac{\psi'_r(t)}{t^{p-1}} = \sup_{s\in (0, t]} \frac{\phi_{B_r}^-(s)}{s^p}$ is increasing and positive
so that $\psi_r'$ satisfies \inc{p-1} and $\psi_r$ is convex, strictly increasing and satisfies \inc{p}. 
Since $\phi^-_{B_r}$ satisfies \ainc{p} and \adec{}, the integrand is finite in every point, and so 
$\psi_r$ is continuous. Thus $\psi_r \in \Phic$. 
Further, by \adec{},
\[
\psi_r(t) \ge  \int_{t/2}^t \tau^{p-1}\sup_{s\in (0, \tau]} \frac{\phi_{B_r}^-(s)}{s^p} \, d\tau
\gtrsim \frac{\phi_{B_r}^-(t/2)}{(t/2)^p} \Big( \frac{t}2 \Big)^p \approx \phi_{B_r}^-(t),
\]
and since $\phi^-_{B_r}$ satisfies \ainc{p} we obtain
\[
\psi_r(t) \lesssim \int_0^t \tau^{p-1} \frac{\phi_{B_R}^-(t)}{t^p} \, d\tau = \tfrac1p \phi_{B_r}^-(t).
\]
Thus $\psi_r(t)\approx \phi_{B_r}^-(t)$. 
It follows that $\psi_r$ satisfies \adec{q}, and since it is convex, 
it satisfies \dec{} \cite[Lemma~2.2.6]{HarH19}. 
Therefore, $\psi_r$ satisfies the assumptions of Lemma~\ref{lem:caccioppoli} with
$p_1=p$; $q_1$ is a function of $q$, but it does not affect the constant. 
See Proposition~\ref{prop:vBound} for a simpler, sufficient condition for \eqref{eq:vBound}. 

\begin{thm}\label{thm:essinf}
Suppose $\phi\in \Phiw(\Omega)$ satisfies \azero{}, \ainc{p} and \adec{q}, $1\le p\le q$. 
Let $u$ be a nonnegative supersolution to \eqref{eq:A-eq} in $B_R$. 
Let $\omega \in \Phi_w(\Omega)$ satisfy \azero{} and \adec{}. 
Assume that $\phi$ satisfies \aones{\omega}, and 
\begin{equation}\label{eq:vBound}
\omega_{B_R}^-\bigg(\fint_{B_R} \frac{u+R}R\, dx\bigg) \le \frac d{|B_R|}
\end{equation}
for some $d>0$. For any $\ell > 0$, there exists a constant $C_{\ell,d} = C(p, q, L_p, L_q, n, \ell,d)>0$ 
such that
\begin{align*}
\essinf_{B_{R/2}} u+R 
\geq 
C_{\ell,d}\Big(\fint_{B_R} (u+R)^{-\ell} \, dx  \Big)^{-\frac{1}{\ell}}.
\end{align*}
\end{thm}

\begin{proof}
Let us assume that $r\in [\frac{R}{2},R]$ and 
denote $ \tilde u := u+R, v:= \frac{\tilde u}{r}$ and $n' := \frac{n}{n-1}$. 
Let $\psi_r \in \Phic$ be as in Definition~\ref{def:psi} in the ball $B_r$, 
and abbreviate $\psi := \psi_r$.
Let  $s>0$ be a constant that will be fixed later. 
We use the $W^{1,1}$-Sobolev--Poincar\'e inequality for the function $\psi(v)^{-\ell} \eta^{s}$, 
where $\eta \in C_0^\infty (B_r)$ is a cut-off function as before. 
We see that $\psi(v)^{-\ell}\eta^{s} \in W^{1,1}(\Omega)$  and it has  a compact support in $\Omega$ by the same arguments as in the proof of the Caccioppoli inequality (Lemma~\ref{lem:caccioppoli}). 
The Sobolev--Poincar\'e inequality gives us
\begin{equation*}
\begin{split}
\Big (\int_{B_r} \psi(v)^{-\ell n'} \eta^{sn'} \, dx \Big )^{\frac{1}{n'}} 
&\lesssim \int_{B_r} |\nabla(\psi(v)^{-\ell} \eta^{s})| \, dx \\
&\le \int_{B_r} \psi(v)^{-\ell-1} \eta^{s-1} [ s\psi(v) |\nabla \eta|  + \ell \eta |\nabla \psi(v)| ]\, dx.
\end{split}
\end{equation*}
By the definition of $\psi$ and \ainc{p} of $\phi$, 
$\psi'(t) \approx \frac1t \phi^-_{B_r} (t) \le \frac1t \phi(x,t)$ for $x \in B_r$. Thus 
\[
|\nabla \psi(v)| 
= 
|\psi'(v) \tfrac{1}{r} \nabla \tilde u|
\lesssim
\tfrac{1}{rv} \phi(x,v) |\nabla \tilde u| 
\] 
almost everywhere in $B_r$. We use this and the estimate $|\nabla \eta| \leq \frac{2}{(1-\sigma)r}$:
\begin{align}
\label{eq:after-sob}
 \Big(\int_{B_r} \psi(v)^{-\ell n'} \eta^{sn'} \, dx  \Big)^{\frac{1}{n'}}  
 &\leq 
\frac {C(s+\ell)}{r} 
\int_{B_r} \psi(v)^{-\ell-1} \eta^{s-1} 
\big[\tfrac{\psi(v)}{1-\sigma}  + \eta \tfrac{1}{v}  \phi(x,v) |\nabla \tilde u| \big]\, dx .
\end{align}

By Young's inequality \cite[(2.4.1)]{HarH19} and $\phi^*(x,\frac{1}{t}\phi(x,t))\le \phi(x,t)$ 
\cite[p.~35]{HarH19}, we have 
\[
\tfrac{1}{v} \phi(x,v) |\nabla \tilde u| 
\le \phi^*\big(x,\tfrac{1}{v} \phi(x,v)\big) + \phi(x,|\nabla \tilde u|)
\le 
\phi(x,v) + \phi(x,|\nabla \tilde u|). 
\]
This and the Caccioppoli inequality (Lemma \ref{lem:caccioppoli}) for $u+R-r$ in $B_r$ yield
\[
\begin{split}
\int_{B_r} \psi(v)^{-\ell-1} \eta^{s} \tfrac{1}{v}\phi(x,v) |\nabla \tilde u| \, dx
&\le   \int_{B_r} \psi(v)^{-\ell-1} \eta^{s} \big(\phi(x,v) + \phi(x,|\nabla \tilde u|)\big) \, dx\\
&\le 
\frac C{(1-\sigma)^q} \int_{B_r} \psi(v)^{-\ell-1} \eta^{s-q}  \phi(x,v)\, dx,
\end{split}
\]
where we assumed that $(\ell+1)p >1$ and $s\ge q$ and used $\eta^{s} \le \eta^{s-q}$. 
We next divide \eqref{eq:after-sob} by $r^{n-1}$, use this estimate as well as  
$\eta^{s-1} \le \eta^{s-q}$ and $\psi(t) \lesssim \phi(x,t)$:
\begin{align*} 
\Big(\fint_{B_r} \psi(v)^{-\ell n'} \eta^{sn'} \, dx  \Big)^{\frac{1}{n'}}  
&\leq 
\frac{C(s+\ell)}{(1-\sigma)^q} 
\fint_{B_r} \psi(v)^{-\ell-1} \eta^{s-1}\psi(v)  + \psi(v)^{-\ell-1} \eta^{s-q}  \phi(x,v)\, dx\\
&\lesssim
\frac {s+\ell}{(1-\sigma)^q} 
\fint_{B_r} \psi(v)^{-\ell-1} \eta^{s-q}  \phi(x,v)\, dx.
\end{align*}
 
Let us denote $V_R:= (\omega_{B_R}^-)^{-1}(d/|B_R|)$ and $E:= \{v(x) < V_R\}$. 
Since $v\ge 1$, we find by \azero{}, \adec{} and \aones{\omega} that $\phi(x, v) \approx \psi(v)$ 
in $E$. Hence
\[
\frac{1}{|B_r|}\int_{B_r \cap E} \psi(v)^{-\ell-1} \eta^{s-q} \phi(x,v)\, dx 
\approx 
\frac{1}{|B_r|} \int_{B_r \cap E} \psi(v)^{-\ell} \eta^{s-q} \, dx.
\]
Since $\phi$ satisfies \adec{q} and $\psi$ satisfies \ainc{p}, we see that 
$t\mapsto \psi(t)^{-\ell-1}\phi(x,t)$ is almost decreasing when $(\ell+1)p\ge q$. 
Then when $v\ge V_R$ we obtain that 
\[
\begin{split}
\frac{1}{|B_r|}\int_{B_r \setminus E} \psi(v)^{-\ell-1} \eta^{s-q} \phi(x,v)\, dx 
&\lesssim 
\frac{1}{|B_r|}\int_{B_r \setminus E} \psi(V_R)^{-\ell-1} \eta^{s-q}  \phi(x, V_R)\, dx\\
&\lesssim 
\fint_{B_r} \psi(V_R)^{-\ell} \eta^{s-q}  \, dx
\lesssim  \psi(V_R)^{-\ell},
\end{split}
\]
again by $\phi(x, V_R) \approx \psi(V_R)$ (from \aones{\omega}). Furthermore, by \eqref{eq:vBound} 
we obtain
\[
\begin{split}
\fint_{B_{r/2}} v \, dx 
\lesssim  
\fint_{B_{R}} \frac{u+R}{R} \, dx
\lesssim  
(\omega^-_R)^{-1}\Big(\frac{d}{|B_R|} \Big) = V_R.
\end{split}
\]
This and Jensen's inequality for the convex function $t\mapsto \psi(t)^{-\ell}$ yield
\[
\psi(V_R)^{-\ell}
\lesssim 
\psi\bigg( \fint_{B_{r/2}} v\, dx\bigg)^{-\ell}
\lesssim 
\fint_{B_{r/2}} \psi(v)^{-\ell} \, dx
\lesssim 
\fint_{B_r} \psi(v)^{-\ell} \eta^{s-q} \, dx.
\]
Combining the estimates in $E$ and $B_r\setminus E$ and the previous inequality, we have established that 
\begin{equation}\label{eq:splitEstimate}
\Big(\fint_{B_r} \psi(v)^{-\ell n'} \eta^{sn'} \, dx  \Big)^{\frac{1}{n'}}  
\leq 
\frac {C(s+\ell)}{(1-\sigma)^{q}}
\fint_{B_r} \psi(v)^{-\ell } \eta^{s-q} \, dx  .
\end{equation}

Let us choose $s:=\ell-(n-1)q$, and suppose that $\ell$ is so large that $s \ge q$, 
$(\ell+1)p\ge q$ and $\ell\ge nq$. 
Raising both sides of the previous inequality to the power $-\frac1\ell$ gives 
\begin{align*}
\underbrace{\Big(\fint_{B_r} \psi(v)^{-\ell n'} \eta^{\ell n' - nq} \, dx  \Big)^{-\frac{1}{n' \ell}}}
_{=:\Psi(n'\ell)}
\ge 
\Big(\frac {C\ell}{(1-\sigma)^{q}} \Big)^{-\frac1\ell}
\underbrace{\Big(\fint_{B_r} \psi(v)^{-\ell} \eta^{\ell - nq} \, dx \Big)^{-\frac1\ell}}
_{=\Psi(\ell)}.
\end{align*}
Let us then set $\ell=n_k:=(n')^k$. For $k\ge k_0$ (so that the required 
lower bounds on $\ell$ hold)  
we use the standard iteration technique. By induction, we obtain that 
\[
\Psi(n_{K}) 
\ge
\exp\Big( -\sum_{k=k_0}^{K-1} \frac{\ln n_k}{n_k}\Big)(C(1-\sigma))^{-\sum_{k=k_0}^{K-1}\frac{q}{n_k}} 
\Psi(n_{k_0}). 
\]
Denote $\gamma:=n_{k_0}$. 
Since $\sum_{k=k_0}^\infty \frac{q}{n_k} = \frac{nq}\gamma =:\beta<\infty$ as a geometric series and 
$\sum_{k=k_0}^\infty \frac{\ln n_k}{n_k} < \infty$ by comparison with a geometric series, we get
\begin{align*}
\essinf_{B_{\sigma r}} \psi(v) 
&\ge 
\essinf_{B_{r}} \frac{\psi(v)}\eta 
= \lim_{K\to \infty} \Psi(n_{K}) \\ 
&\gtrsim
(1-\sigma)^{-\beta} \Psi(\gamma) 
= 
(1-\sigma)^{-\beta} \bigg(\fint_{B_r} \psi(v)^{-\gamma} \eta^{\gamma - nq}\, dx \bigg)^{-\frac1\gamma} \\
&\ge 
(1-\sigma)^{-\beta} \bigg(\fint_{B_{r}} \psi(v)^{-\gamma} \, dx \bigg)^{-\frac1\gamma}
= 
 (1-\sigma)^{-\beta} \psi\bigg(\xi^{-1}\Big(\fint_{B_{r}} \xi(v)\, dx \Big)\bigg)\\
&\gtrsim  \psi\bigg( (1-\sigma)^{-\frac\beta p} \xi^{-1}\Big(\fint_{B_{r}} \xi(v)\, dx \Big)\bigg),
\end{align*}
where $\xi(t):=\psi(t)^{-\gamma}$, and \ainc{p} of $\psi$ was used in the last inequality. 
As $\psi$ is strictly increasing, this implies
\[
\essinf_{B_{\sigma r}} v \gtrsim (1-\sigma)^{-\frac\beta p} \xi^{-1}\Big(\fint_{B_{r}} \xi(v)\, dx \Big).
\]
Since $t\mapsto \xi(t^{-1/(\gamma q)})$ satisfies \adec{1}, it is equivalent to 
a concave function \cite[Lemma~2.2.1]{HarH19}, and so by Jensen's inequality 
\begin{align*}
\essinf_{B_{\sigma r}} v
\gtrsim (1-\sigma)^{-\frac\beta p}
\xi^{-1}\Big(\fint_{B_{r}} \xi(v)\, dx \Big)
&\gtrsim(1-\sigma)^{-\frac\beta p}
\Big( \fint_{B_{r}} v^{-\gamma q} \, dx  \Big)^{- \frac{1}{\gamma q}}.
\end{align*}
Then we recall that  $v= \frac{u+R}{r}$ and multiply both sides by $r$: 
\begin{align}\label{equ:suuret-arvot}
\essinf_{B_{\sigma r}} u + R 
&\gtrsim  
(1-\sigma)^{-\frac\beta p} \Big( \fint_{B_r} (u + R)^{-\gamma q} \, dx  \Big)^{- \frac{1}{\gamma q}}, 
\end{align}
where $\frac{R}{2}\le r \le R$ and the constant depends only on $p$, $q$, $L_p$, $L_q$, $n$, $\ell$ and $d$. 
This is the claim for $\ell=\gamma q$. 
For exponents larger than $\gamma q$ the claim follows by Hölder's inequality.
We have thus established the claim for large exponents. 

Finally we show the claim for small exponents. We use the $+R$ to simplify the proof, but this 
is not essential here. So let $\ell\in (0,\gamma q)$. We observe that 
\begin{align*}
\essinf_{B_{\sigma r}} u + R
&\ge
C (1-\sigma)^{-\frac\beta p} \Big( \fint_{B_r} (u + R)^{-\gamma q} \, dx  \Big)^{- \frac{1}{\gamma q}} \\
& \ge 
C (1-\sigma)^{-\frac\beta p} \Big( \fint_{B_r} (u + R)^{-\ell} \, dx  \Big)^{- \frac{1}{\gamma q}}
(\essinf_{B_r} u+R)^{\frac{\gamma q-\ell}{\gamma q}}.
\end{align*}
We denote $Q:=\frac{\gamma q-\ell}{\gamma q}$, $\Lambda(r):=\essinf_{B_r} u+R$ and 
$A:= ( \fint_{B_R} (u + R)^{-\ell} \, dx )^{-1/\ell}$. Thus, for $r\in [\frac R{2}, R]$, 
\[
\Lambda(\sigma r) \ge C (1-\sigma)^{-\frac\beta p} A^{\frac \ell{\gamma q}} \Lambda(r)^Q. 
\]
Then we set $r:=(1-2^{-(k+1)})R$ and $\sigma r := (1-2^{-k})R$ so that $1-\sigma \approx 2^{-k}$. 
With $\Lambda_k:=\Lambda((1-2^{-k})R)$ and iteration we obtain 
\[
\begin{split}
\Lambda_1 
&\ge 
C 2^{\frac\beta p} A^{\frac \ell{\gamma q}} \Lambda_2^{Q}
\ge 
2^{\frac\beta p(1+2Q)} (CA^{\frac \ell{\gamma q}})^{1+Q} \Lambda_3^{Q^2}
\ge \ldots \\
&\ge 
2^{\frac\beta p \sum_{k=1}^\infty kQ^{k-1}} (CA^{\frac \ell{\gamma q}})^{\sum_{k=0}^\infty Q^k} \liminf_{k \to \infty} \Lambda_k^{Q^k}.
\end{split}
\]
Since $\Lambda_k\ge R$ and $Q\in (0,1)$, $ \liminf_{k \to \infty} \Lambda_k^{Q^k}\ge 1$. Furthermore, 
$\sum_{k=1}^\infty kQ^{k-1} < \infty$ and $\sum_{k=0}^\infty Q^k = \frac 1{1-Q} = \frac{\gamma q}\ell$. 
Hence $\Lambda_1 \gtrsim A$, which is the claim for $\ell$. 
\end{proof}

The previous proof only works for negative $\ell$. In fact, the largeness of $-\ell$ and the condition
\aones{\omega}  were only used in the paragraph with \eqref{eq:splitEstimate}. 
With some modifications, we can iterate also for some positive exponents. 
We use the limiting exponent $\ell(p)$ defined in the introduction. 
See Proposition~\ref{prop:betaBound} for a simpler, sufficient condition for \eqref{eq:betaBound}. 

\begin{prop}\label{prop:betaCondition}
Suppose $\phi\in \Phiw(\Omega)$ satisfies \azero{}, \ainc{p} and \adec{q}, $p,q>1$. 
Let $u$ be a nonnegative supersolution to \eqref{eq:A-eq} in $B_R$. 
Assume that 
\begin{equation}\label{eq:betaBound}
\fint_{B_r} \Big(\frac{\phi(x,v)}{\phi_{B_r}^-(v)}\Big)^{\beta}\, dx \le d,
\end{equation}
for some $\beta>\max\{\frac np,1\}$ and all $B_r\subset B_R$, where $v:=\frac{u+r}r$. 
For any $\ell_0>0$ and $\ell< \ell(p)$, 
there exists a constant $C = C(p,q, L_p,L_q,n, \ell_0, \ell, d)>0$ such that
\begin{align*}
\Big(\fint_{B_{2R}} (u+R)^{\ell_0} \, dx  \Big)^{\frac{1}{\ell_0}}
\ge
C \psi^{-1}\bigg(\Big(\fint_{B_R} \psi(v)^{\frac \ell p} \, dx  \Big)^{\frac{p}{\ell}} \bigg)
\ge
C\Big(\fint_{B_R} (u+R)^{\ell} \, dx  \Big)^{\frac{1}{\ell}}.
\end{align*}
\end{prop}

\begin{proof}
We proceed as in the previous proof, but use the $W^{1,\gamma}$-Sobolev inequality 
instead of the $W^{1,1}$-version. Here we will eventually take $\gamma \nearrow \min\{p,n\}$. 
In place of \eqref{eq:after-sob} we obtain 
\[
\begin{split}
\Big(\fint_{B_r} \psi(v)^{-\ell \gamma^*} \eta^{s \gamma^*} \, dx  \Big)^{\frac{\gamma}{\gamma^*}}  
& \lesssim
\fint_{B_r} \big(\psi(v)^{-\ell-1} \eta^{s-1} 
\big[\tfrac{\psi(v)}{1-\sigma}  + \eta  \psi'(v) |\nabla \tilde u| \big]\big)^\gamma\, dx \\
&\lesssim \fint_{B_r} \psi(v)^{-\ell \gamma}  \tfrac{\eta^{\gamma(s-1)}}{(1-\sigma)^\gamma}  + \eta^{s \gamma}  \psi(v)^{-\ell \gamma -1} \tfrac{\psi(v)}{v^\gamma} |\nabla \tilde u|^\gamma\, dx 
\end{split}
\]
where we already divided by $r^{n-1}$, and used Remark~\ref{rem:incdec} as well as
$(a+b)^\gamma \approx a^\gamma+b^\gamma$. 

We estimate $\frac{\psi(v)}{v^\gamma} |\nabla \tilde u|^\gamma$ with Young's inequality. 
Define $\xi(t):=\psi(t^{1/\gamma})$. Then $\xi^{-1}(t) = \psi^{-1}(t)^\gamma$ and 
\[
(\xi^*)^{-1}(t) \approx \frac{t}{\xi^{-1}(t)} = \frac{t}{\psi^{-1}(t)^\gamma}
\]
since $\xi^{-1}(t)(\xi^*)^{-1}(t)\approx t$ by \cite[Theorem~2.4.8]{HarH19}. Hence 
\[
\xi^*\Big(\frac{\psi(t)}{t^\gamma} \Big) \approx \psi(t)
\]
and so by Young's inequality
\[
\frac{\psi(v)}{v^\gamma} |\nabla \tilde u|^\gamma 
\le
\xi(|\nabla \tilde u|^\gamma) + \xi^*\Big( \frac{\psi(v)}{v^\gamma} \Big) 
\approx 
\psi(|\nabla \tilde u|) + \psi(v)
\lesssim
\phi(x,|\nabla \tilde u|) + \psi(v).
\]

Then we estimate with the Caccioppoli inequality (Lemma~\ref{lem:caccioppoli})
\[
\fint_{B_r} \psi(v)^{-\ell \gamma -1} \eta^{s \gamma} \phi(x,|\nabla \tilde u|)\, dx 
\lesssim 
\tfrac{1}{(1-\sigma)^q} \fint_{B_r} \psi(v)^{-\ell \gamma-1} \eta^{s\gamma -q} \phi(x,v) \, dx 
\]
provided $\ell \gamma +1> \frac 1p$, which means that $\ell$ can also be negative. 
Thus we have
\[
\begin{split}
\Big(\fint_{B_r} \psi(v)^{-\ell \gamma^*} \eta^{s \gamma^*} \, dx  \Big)^{\frac{\gamma}{\gamma^*}}  
&\lesssim 
\fint_{B_r}  \psi(v)^{-\ell \gamma}  \tfrac{\eta^{\gamma(s-1)}}{(1-\sigma)^\gamma} 
+ \psi(v)^{-\ell \gamma-1} \tfrac{\eta^{s\gamma-q}}{(1-\sigma)^q} \phi(x,v) 
+ \eta^{s \gamma}  \psi(v)^{-\ell \gamma} \, dx \\
&\le  \tfrac{1}{(1-\sigma)^q} \fint_{B_r}  \psi(v)^{-\ell \gamma-1} \eta^{s-q} \phi(x,v)  \, dx,
\end{split}
\]
where $\psi(v) \lesssim \phi(x, v)$,  $\eta^{\gamma (s-1)} \le \eta^{s-q}$ and $\eta^{\gamma s} \le \eta^{s-q}$ were used. 
In contrast to the previous proof, we next use H\"older's inequality 
\[
\fint_{B_r} \psi(v)^{-\ell\gamma-1} \eta^{s-q}  \phi(x,v)\, dx
\le
\bigg(\fint_{B_r} \psi(v)^{-\ell \gamma \lambda} \eta^{(s-q)\lambda}  \, dx \bigg)^\frac1\lambda
\bigg(\fint_{B_r} \Big(\frac{\phi(x,v)}{\psi(v)}\Big)^{\lambda'}\, dx \bigg)^\frac1{\lambda'}.
\]
Since the second factor on the right-hand side is bounded by \eqref{eq:betaBound} when 
$\lambda' \in (\frac{n}{\gamma}, \beta]$, we obtain that
\[
\Big(\fint_{B_r} \psi(v)^{-\ell \gamma^*} \eta^{s\gamma^*} \, dx  \Big)^{\frac{1}{\gamma^*}}  
\lesssim
\bigg(\fint_{B_r} \psi(v)^{-\ell \gamma \lambda} \eta^{(s-q)\lambda} \, dx\bigg)^\frac1{\gamma\lambda}.
\]
Since $\lambda < (\frac{n}{\gamma})' = \frac n{n-\gamma}$, we conclude that 
$\gamma \lambda < \gamma^*$. 
Therefore, we have obtained a reverse H\"older 
type inequality, which can be iterated (as in the previous proof) to show that 
\[
\bigg(\fint_{B_{2r}} v^{\ell_0} \, dx\bigg)^\frac1{\ell_0}
\gtrsim 
\psi^{-1}\Bigg(\Big(\fint_{B_r} \psi(v)^{\gamma_2} \, dx  \Big)^{\frac{1}{\gamma_2}} \Bigg)
\ge
\bigg(\fint_{B_{r}} v^{p\gamma_2} \, dx\bigg)^\frac1{p\gamma_2}
\]
for $\gamma_2\le -\ell \gamma^*$ and any $\ell_0>0$; the last step is just Jensen's inequality. 
Since we can choose any value of $\ell$ with $-\ell < \frac1{\gamma p'}$, we have the inequality 
for any $\gamma_2 < \frac{\gamma^*}{\gamma p'} = \frac n{n-\gamma} \frac{p-1}p$. 
Thus letting $\gamma \nearrow \min\{p,n\}$ we can obtain the weak Harnack inequality for exponent up to, but not including, $\ell(p)$. 
\end{proof}



We are now ready for the proof of the main result, the weak Harnack inequality. 
Note that here we need to add the \ainc{} assumption for $\phi$ compared to Theorem~\ref{thm:essinf}. 

\begin{thm}\label{thm:weakHarnackGeneral}
Suppose $\phi$ satisfies \azero{}, \ainc{p} and \adec{q}, $1<p\le q<\infty$. 
Let $u$ be a non-negative solution to \eqref{eq:A-eq} on $B_{2R}$.
Assume that there exists $\omega \in \Phiw(\Omega)$ which satisfies \azero{}  and \adec{} such that 
one of the following holds:
\begin{enumerate}
\item 
$\phi$ satisfies \aones{\omega} and \eqref{eq:vBound} and \eqref{eq:betaBound}, with $\beta=1$, hold.
\item 
$\phi$ satisfies \aones{\omega} and $\| u\|_{W^{1,\omega}(B_{R})} \le d$.
\end{enumerate}  
Then there exist $\ell_0=\ell_0(p,q,L_p,L_q,\beta,d,n)>0$ and  $C = C(p, q, L_p,L_q,\beta, d, n)>0$ such that
\begin{align*}
\bigg (\fint_{B_{2R}} (u +R)^{\ell_0} \, dx \bigg)^{\frac{1}{\ell_0}}  \le C (\essinf_{B_R} u + R).
\end{align*}
If \eqref{eq:betaBound} holds with $\beta>\max\{\frac np,1\}$, then we can choose any $\ell_0<\ell(p)$. 
\end{thm}

\begin{proof}
Let $0< r \le \frac12 R$ and denote $v := \frac{u +r}r$ so that $v \approx \frac{u+2r}{2r}$. 
By the $W^{1,1}$-Poincar\'e inequality we get
\begin{align*}
\fint_{B_{r}} \Big|\log(u+R) - \fint_{B_r}\log(u+R)\, dy\Big| \, dx 
\leq 
r \fint_{B_r} |\nabla (\log(u+R))| \, dx 
\approx \fint_{B_r} \frac{|\nabla u|}{v} \, dx.
\end{align*}
Considering the cases $ \frac{|\nabla u|}{v}  \le 1$ and $\frac{|\nabla u|}{v} >1$, 
we conclude that 
\[
\frac{|\nabla u|}{v} 
\lesssim 
1 + \frac{\phi_{B_{2 r}}^-(|\nabla u|)}{\phi_{B_{2 r}}^-(v)}
\lesssim
1 + \frac{\phi(x,|\nabla u|)}{\psi_{2 r}(v)},
\]
where \ainc{1} was used in the case $\frac{|\nabla u|}{v} >1$, and $\psi_{2 r} \approx \phi^-_{B_{2 r}}$ in the last step. 
Note that $\psi_{2 r}$ satisfies \inc{p}.
It follows from the Caccioppoli inequality (Lemma~\ref{lem:caccioppoli}) 
with $\ell=1> \frac1{p}$ and $s=q$ that 
\[
\begin{split}
\int_{B_r} \frac{\phi(x,|\nabla u|)}{\psi_{2 r}(v)} \, dx 
&\lesssim
\int_{B_{2r}} \frac{\phi(x,|\nabla u|)}{\psi_{2 r}\big(\frac{u+2r}{2r}\big)} \eta^q \, dx 
\lesssim 
\int_{B_{2r}}  \frac{\phi\big(x, \frac{u+2r}{2r}\big)}{\psi_{2 r}\big(\frac{u+2r}{2r}\big)} \, dx 
\approx
\int_{B_{2r}}  \frac{\phi\big(x, \frac{u+2r}{2r}\big)}{\phi^-_{2 r}\big(\frac{u+2r}{2r}\big)} \, dx.
\end{split}
\]
Then we divide by $|B_{r}|$, note that $|B_{r}|\approx |B_{2r}|$, and combine with the previous inequalities to obtain that 
\[
\fint_{B_r} \frac{|\nabla u|}{v} \, dx 
\lesssim 
1+
\fint_{B_r} \frac{\phi(x,|\nabla u|)}{\psi_{2 r}(v)} \, dx 
\lesssim 
1+
\fint_{B_{2r}}  \frac{\phi\big(x, \frac{u+2r}{2r}\big)}{\phi^-_{2 r}\big(\frac{u+2r}{2r}\big)} \, dx 
\le 1+d,
\]
where \eqref{eq:betaBound} have been used in the last inequality.
Thus we have established $\log(u+R) \in BMO$ under assumption  (1). 

Next we consider assumption (2) with $\|\nabla u\|_{L^\omega(B_{2R})} \le d$.
Define 
\begin{align*}
E &:= \{x \in B_{2r} : \omega^{-1}(x,\tfrac1{r^n}) < v(x)\}.
\end{align*}
In $B_r\setminus E$, we have $\phi^+_{B_{2r}}(v)\approx \phi^-_{B_{2r}}(v)$ by the \aones{\omega} 
condition of $\phi$ and $v \ge 1$. Then we use the Caccioppoli inequality as before, except instead of 
$\psi_r$ we use a corresponding function with $\psi_{2r} \approx \phi^+_{B_{2r}}$ and so 
do not need \eqref{eq:betaBound}. 

In $E$, we use Young's inequality for the $\Phi$-function $\xi:=r^n \omega$: 
\[
\frac{|\nabla u|}{v} 
\le 
\xi(x,\tfrac1\epsilon|\nabla u|) + \xi^*(x, \tfrac{\epsilon}{v})
=
r^n \omega(x,\tfrac1\epsilon|\nabla u|) + r^n\omega^*(x, \tfrac{\epsilon}{r^{n}v}), 
\]
since $\xi^* (t) = r^n \omega^*(t/r^n)$  by \cite[Lemma 2.4.3]{HarH19}.
For the first term, $\fint_{B_r} r^n \omega(x,\tfrac1\epsilon|\nabla u|)\, dx \le c$ by the assumption 
$\|\nabla u\|_{L^\omega(B_{2R})} \le d$ and \adec{} of $\omega$. 
In $E$, $v > \omega^{-1}(x,\tfrac1{r^n}) \approx \frac{1}{r^n (\omega^*)^{-1}(x,\frac1{r^n})}$, 
since $\omega^{-1}(t) (\omega^*)^{-1}(t) \approx t$ \cite[Theorem 2.4.8]{HarH19}. 
So $\frac1{r^{n}v} \lesssim (\omega^*)^{-1}(x,\frac1{r^n})$, and hence 
$\omega^*\big(x, \frac\epsilon{r^{n}v}\big) \le \frac1{r^n}$ for appropriate $\epsilon>0$. 
Therefore
\begin{align*}
\frac{1}{|B_{2r}|}\int_{E} r^n\omega^*(x, \tfrac{\epsilon}{r^{n}v}) \, dx 
\le
\frac{1}{|B_{2r}|} \int_{E} 1 \, dx \leq 1.
\end{align*}
We have bounded all terms, so $\log(u+R) \in BMO$ also under assumptions (2).

Since $\log(u+R)\in BMO$, the John--Nirenberg lemma says that there exist positive 
constants $\ell_0$ and $C$ such that
\begin{align*}
\bigg (\fint_{B_{2R}}  (u+R)^{\ell_0} \, dx  \bigg) \bigg(\fint_{B_{2R}} (u+R)^{-\ell_0} \, dx  \bigg) \leq C.
\end{align*}
Let us note that under assumption (2), condition \eqref{eq:vBound} holds by 
the Poincar\'e inequality and Jensen's inequality, see Proposition~\ref{prop:vBound}.
By the previous inequality and Theorem~\ref{thm:essinf} we obtain
\[
\bigg (\fint_{B_{2R}} (u+R)^{\ell_0} \, dx  \bigg)^{\frac{1}{\ell_0}} 
\lesssim 
\bigg (\fint_{B_{2R}} (u+R)^{-\ell_0} \, dx  \bigg)^{-\frac{1}{\ell_0}}  
\lesssim 
\essinf_{B_R} u + R .  
\]

Assume then that \eqref{eq:betaBound} holds with $\beta>\max\{\frac np,1\}$.  
Then by Proposition~\ref{prop:betaCondition} we  have
for any $\ell_0>0$ and $\ell< \ell(p)$,  that
\begin{align*}
\Big(\fint_{B_{4R}} (u+2R)^{\ell} \, dx  \Big)^{\frac{1}{\ell}} 
\lesssim 
\Big(\fint_{B_{2R}} (u+2R)^{\ell_0} \, dx  \Big)^{\frac{1}{\ell_0}},
\end{align*}
and hence the claim follows when we combine this with the previous inequality. 
(With this proof we obtained the ball $B_{4R}$ on the right-hand side. Proving the 
intermediate results for $B_{\sqrt 2 R}$ give $B_{2R}$.) 
\end{proof}


\section{Special cases}\label{sect:special}

Let us next investigate conditions \eqref{eq:vBound} and \eqref{eq:betaBound}. 
We define the ``Sobolev conjugate'' $\omega^\#\in\Phiw$ by the condition 
$(\omega^\#)^{-1}(t) := t^{-1/n} \omega^{-1}(t)$ for $t>0$. Note that for $\omega^\#$ 
to be in $\Phiw$, it must be increasing and so $\omega$ must satisfy \dec{n}, see \cite[Lemma~5.2.3]{HarH19}.
If $\omega(t)=t^s$, then $\omega^\#(t)=t^{s^*}$, where $s^*:=\frac{ns}{n-s}$ is 
the Sobolev exponent. Also $s=n$ is included and in this case $\|u\|_{\omega^\#} = \|u\|_\infty$. 
Thus we must not assume that $\omega^\#$ satisfies \adec{}, which means that extra care 
must be taken in the next proof regarding constants and inverse functions, but tools for this were 
established in \cite{HarH19}.

Note that (2) implies (1) if we have a Sobolev inequality, but then we have difficulties 
with the case $s\ge n$ in (2), so we provide separate proofs for the two cases. 

\begin{prop}\label{prop:vBound}
Let $B_R \subset \Omega$, $R \le 1$ and
let $\omega\in \Phiw(\Omega)$ satisfy \azero{} and \adec{}. 
Assume that $\phi\in \Phiw(\Omega)$ satisfies \aones{\omega} and that one of the 
following hold:
\begin{enumerate}
\item
$u\in L^{\omega^\#}(B_R)$ with $\|u\|_{\omega^\#}\le d$, where $\omega^\#\in \Phiw(\Omega)$.
\item
$u\in W^{1,\omega}(B_R)$ with $\|u\|_{1,\omega}\le d$.
\end{enumerate}
Then \eqref{eq:vBound} holds, for $v:=\frac{u+R}R$, i.e.\
\[
\omega_{B_R}^-\bigg(\fint_{B_R} v \, dx \bigg) \le \frac {d_2}{|B_R|}. 
\]
\end{prop}
\begin{proof}
By an elementary embedding, $\|u\|_{W^{1,\omega_{B}^-}(B)}\lesssim \|u\|_{W^{1,\omega}(\Omega)}$. 
Therefore it suffices to prove the result for $\omega\in \Phiw$ (i.e.\ $\omega$ independent of $x$) 
and apply this result to $\omega_ {B_R}^-$, and similarly in the case of assumption (1). 

Let us first use assumption (1). 
Jensen's inequality \cite[Lemma 4.3.1]{HarH19}, and \azero{} of $\omega^\#$ yield that, for sufficiently 
small $d'>0$, 
\[
\begin{split}
\omega^\#\bigg(2d'\fint_{B_R} \tfrac{u+R}2\, dx\bigg) 
&\le \fint_{B_R} \omega^\#(\tfrac{u+\beta_0}2)\, dx  
\lesssim 
|B_R|^{-1}+1 \lesssim |B_R|^{-1}
\end{split}
\]
with $\beta_0$ from \azero{}. 
By \cite[Lemma~2.3.9]{HarH19}, $(\omega^\#)^{-1}\big(\omega^\#(t)\big)\approx t$ or $\omega^\#(t)=0$. 
In the former case, 
\[
\fint_{B_R} u+R\, dx 
\lesssim 
(\omega^\#)^{-1}(c \,|B_R|^{-1})
=
c' R \omega^{-1}(c \,|B_R|^{-1})
\approx 
R \omega^{-1}(|B_R|^{-1}).
\]
Since $\omega$ satisfies \adec{}, we can move the constants outside.
We divide both sides by $R$ and apply $\omega$ to obtain 
$\omega(\fint_{B_R} v\, dx)\lesssim |B_R|^{-1}$, i.e.\ \eqref{eq:vBound}. 
In the case $\omega^\#(t)=0$ we conclude from \azero{} and \dec{n} of $\omega$ that 
$t\le d''$ and so 
\[
\omega\bigg(\fint_{B_R} v\, dx\bigg) 
\le 
\omega\bigg(\frac{1}{d' R} 2d'\fint_{B_R} \tfrac{u+R}2\, dx\bigg) 
\le 
\omega\bigg(\frac{d''}{d' R}\bigg) 
\lesssim 
R^{-n}
\]
also by \dec{n} of $\omega$. Thus we have \eqref{eq:vBound} under assumption (1). 

For the other case, we first use  Hölder's inequality to obtain
\[
\fint_{B_R} v\, dx \lesssim R^{-n} \Big(\int_{B_R}  (u+R)^{n'}\, dx\Big)^\frac1{n'}.
\]
Then we use that $B_R$ is a $W^{1,1}$-extension domain, 
extend $u+R$ to $\Rn$ and denote this extension by $v$. We obtain by the $W^{1,1}$-Sobolev embedding
\[
\|u+R\|_{L^{n'}(B_R)}
\le 
\|v\|_{L^{n'}(\Rn)}
\lesssim \|v\|_{W^{1,1}(\Rn)}
\lesssim \|u+R\|_{W^{1,1}(B_R)}. 
\]
Combining the above steps, we have
\[
\fint_{B_R} v\, dx \le \fint_{B_R}  u+R+ |\nabla u| \, dx.
\]
Jensen's inequality, \dec{} of $\omega$, \azero{} of $\omega$ , $R \le 1$, and $\|u\|_{1,\omega}\le d$ yield
\[
\begin{split}
\omega\Big( \fint_{B_R} v\, dx \Big) &\le  \omega\Big( \fint_{B_R} u+R+|\nabla u|\, dx \Big) 
\le \fint_{B_R} \omega (u+R+|\nabla u|)\, dx\\
&\lesssim  \fint_{B_R} \omega (u)+\omega(|\nabla u|) + \omega(1)\, dx \lesssim \frac1{|B_R|}.\qedhere
\end{split}
\] 
\end{proof}

Note that in the next proposition if we consider case (2) with $\omega=\phi$, then $s=p$ and 
$s^*>\max\{\frac np,1\} (q-p)$ is equivalent to $p^* > q$, the condition from Theorem~\ref{thm:weakHarnack}. 

\begin{prop}\label{prop:betaBound}
Let $\omega\in \Phiw(\Omega)$ satisfy \azero{}, \adec{}, and \ainc{s}, $s\ge 1$. 
Let $\phi\in \Phiw(\Omega)$ satisfy \azero{}, \aones{\omega} and \adec{} and assume that one of the 
following hold:
\begin{enumerate}
\item
$u\in L^{\omega^\#}(B_R)$ with $\|u\|_{\omega^\#}\le d$, where $\omega^\#\in \Phiw(\Omega)$.
\item
$u\in W^{1,\omega}(B_R)$ with $\|u\|_{1,\omega}\le d$.
\end{enumerate}
If $s^*\ge\beta (q-p)$, then \eqref{eq:betaBound} holds, 
i.e.\ 
\[
\fint_{B_r} \Big( \frac{\phi(x, v)}{\phi_{B_r}^-(v)}\Big)^\beta\, dx \le C
\]
for and $B_r\subset B_R$ and $v:=\frac{u+r}r$. 
\end{prop}
\begin{proof}
Fix $B_r\subset B_R$. 
As in the previous proof, it suffices to consider $\omega\in \Phiw$ independent of $x$. 
Points where $u(x)\le r$ make only a constant contribution to the integral, 
so we may assume that $u>r$ and take $v=\frac ur$ for simplicity.  Here we used \azero{} and \adec{} of $\phi$.

We denote $V_r:=\omega^{-1}(|B_r|^{-1})$ 
and $E:=\{v<V_r\}$. Then $\omega (V_r) \le \frac1{|B_r|}$ and  \aones{\omega} of $\phi$ yields 
\[
\frac1{|B_r|} \int_{B_r\cap E} \Big( \frac{\phi(x, v)}{\phi^-_{B_r}(v)}\Big)^\beta\, dx
\approx 
\frac1{|B_r|} \int_{B_r\cap E} \Big( \frac{\phi(x, v)}{\phi(x,v)}\Big)^\beta\, dx
\le 1;
\]
we used $v>1$ and \azero{} for the lower bound of \aones{\omega}.
This holds in both case (1) and (2), since these assumption were not used yet. 

In $B_r\setminus E$ we have 
\[
\frac{\phi(x, v)}{\phi^-_{B_r}(v)} 
\lesssim 
\frac{\phi(x, V_r)}{\phi^-_{B_r}(V_r)} (\tfrac v {V_r})^{q-p}
\approx 
(\tfrac v {V_r})^{q-p}
\]
by \ainc{p}, \adec{q} and \aones{\omega}. 
Consider first the case $\|u\|_{\omega^\#}\le d$.
By the definition of $(w^\#)^{-1}$, we obtain that $(w^\#)^{-1}$ satisfies \adec{\frac1s - \frac1n} 
and hence $w^\#$ satisfies \ainc{\frac{sn}{n-s}}
and thus also \ainc{\beta(q-p)}. Therefore
$t\mapsto \omega^\#(t^{1/(\beta(q-p))})$ satisfies \ainc{1}, and hence a Jensen-type inequality \cite[Lemma 4.3.1]{HarH19} yields
\[
\begin{split}
\bigg(\fint_{B_r} (\tfrac v{V_r})^{\beta(q-p)}\, dx\bigg)^\frac1{\beta(q-p)}
&\le 
\frac1{r V_r} (\omega^\#)^{-1}\bigg( \fint_{B_r} \omega^\#(u)\, dx \bigg)
\lesssim
\frac{(\omega^\#)^{-1}( |B_r|^{-1} )}{r V_r}\\
&=  \frac{ |B_r|^{1/n} \omega^{-1}( |B_r|^{-1} )}{r V_r} \approx 1
\end{split}
\]
where we used $\rho_{\omega^\#}(u)\le C$ and the definition of $V_r$. 
This completes the estimate in case (1). 

In the case $\|u\|_{1,\omega}\le d$, the H\"older inequality (with $\beta(q-p)\le s^*$), 
the Sobolev inequality and the Jensen inequality (with \ainc{p}) give 
\[
\begin{split}
\bigg(\fint_{B_r} (\tfrac v {V_r})^{\beta(q-p)}\, dx\bigg)^\frac1{\beta(q-p)}
&\le 
\frac1 {V_r} \bigg(\fint_{B_r} (\tfrac ur)^{s^*}\, dx\bigg)^\frac1{s^*}
\lesssim
\frac1 {V_r} \bigg(\fint_{B_r} u^s+|\nabla u|^s\, dx\bigg)^\frac1{s} \\
&\lesssim 
\frac1{V_r} \omega^{-1}\bigg( \fint_{B_r} \omega(u) + \omega(|\nabla u|) \, dx \bigg)
\lesssim
\frac{\omega^{-1}( |B_r|^{-1} )}{V_r}= 1. \qedhere
\end{split}
\]
\end{proof}

We can now state the weak Harnack inequality for the variable exponent case as a corollary. 
For notation and terminology we refer to \cite{DieHHR11}. Note that here we 
can choose any $s>0$. 

\begin{cor}
Let $\phi(x,t):=t^{p(x)}$ be the variable exponent functional 
and let $u$ be a non-negative solution to \eqref{eq:A-eq} on $B_{2R}$.
We assume that $p$ is $\log$-H\"older continuous with constant $c_{\log}$, $1<p^-\le p^+<\infty$, 
and $\|u\|_s \le d$ for some $s\in (0,\infty]$. 
Then for every  $0<\ell_0< \ell(p^-_{B_R})$, there exists 
$C = C(p^-, p^+, \ell_0,c_{\log}, d, n)$ and $R_0$ such that
\begin{align*}
\bigg (\fint_{B_{2R}} (u+R)^{\ell_0} \, dx \bigg)^{\frac{1}{\ell_0}}  \le C (\essinf_{B_R} u + R)
\quad\text{for all } r \le R_0.
\end{align*}
\end{cor}
\begin{proof}
We check that the assumptions of Theorem~\ref{thm:weakHarnackGeneral} are fulfilled. 
Define $s_*:=\frac{ns}{n+s}$ so that $(s_*)^* = s$ and let $\omega(t):=t^{s_*}$.
Now \aones{s_*} reads as $\beta^{p(x)} t^{p(x) - p(y)} \le 1$ for $t \in [1, |B|^{-\frac{1}{s_*}}]$. 
Since $p$ is $\log$-H\"older continuous, this holds (see \cite[Section~7.1]{HarH19} for details). 
Thus by  Proposition~\ref{prop:vBound} (1) condition \eqref{eq:vBound} holds.

Let $\beta:=n >\frac{n}{p^-_{B_R}}$. We  choose $R_0$ so small that $n(p_{B_R}^+ - p_{B_R}^-)< s$ for $R \le R_0$. Then $\omega^\#(t) = t^s$ satisfies \ainc{\beta(p_{B_R}^+ - p_{B_R}^-)}, and hence \eqref{eq:betaBound} holds  by Proposition~\ref{prop:betaBound} (1). Thus the exponent can be chosen up to $\ell(p^-_{B_r})$.
\end{proof}


\section{The double phase case and counter-examples}\label{sect:example}

Let us study the double phase case. Note that when $s=\infty$ in the next result we obtain 
the special case of bounded supersolutions from \cite{BenK_pp}.

\begin{cor}\label{cor:doublePhase-2}
Let $\phi(x,t):=t^p + a(x)t^q$ be the double phase functional 
and let $u$ be a non-negative supersolution to \eqref{eq:A-eq} on $B_{2R}$.
We assume that $a\in C^{0,\alpha}$ and $\|u\|_{s} \le d$, $s\in (0,\infty]$. If 
$\alpha \ge (\tfrac{n}s+1)(q-p)$ and $s\ge q-p$, then 
there exist  positive constants $\ell_0$ and $C$, such that
\begin{align*}
\bigg (\fint_{B_{2R}} (u +R)^{\ell_0} \, dx \bigg)^{\frac{1}{\ell_0}}  \le C (\essinf_{B_R} u + R).
\end{align*}
Furthermore, if $s>\max\{\frac np,1\}(q-p)$, then the inequality holds for any $\ell_0<\ell(p)$. 
\end{cor}

\begin{proof}
Let us show that assumption (1) of Theorem~\ref{thm:weakHarnackGeneral} is fulfilled.
Let $\omega(t):=t^{s_*}$, where $s_*= \frac{ns}{n+s}$.
Let $x,y\in B_R$ and $t\in [1, |B_R|^{-1/s_*}]$. Then 
\[
\phi(x,t) = t^p + a(x) t^q 
\le
t^p + (a(y) + c R^\alpha) t^q 
\le
(1+cR^\alpha t^{q-p} ) \phi(y,t). 
\]
Since $t\lesssim R^{-\frac n{s_*}}$, the coefficient is bounded provided $\alpha -\frac n{s_*}(q-p)\ge 0$. 
Since $\frac n{s_*} = \frac ns + 1$, this is $\alpha \ge (\tfrac{n}s+1)(q-p)$. We have proved that $\phi$ satisfies \aones{s_*}. Now $(\omega^\#)^{-1}(t) = t^{-1/n} \omega^{-1}(t) = t^{-1/n + 1/s_*}= t^{1/s}$, and thus
$\rho_{\omega^\#}(u) = \rho_{L^s}(u) \le d^s$. Hence \eqref{eq:vBound} holds by Proposition~\ref{prop:vBound}. 
Let us then consider \eqref{eq:betaBound}. Since $\omega^\#(t) = t^{s}$, it satisfies \ainc{q-p} provided that $s\ge q-p$ and \ainc{\beta(q-p)} for $\beta>\max\{\frac np,1\}$ provided that $s> \max\{\frac np,1\}(q-p)$. 
Thus \eqref{eq:betaBound} follows by Proposition~\ref{prop:betaBound} and the claims follow by 
Theorem~\ref{thm:weakHarnackGeneral}.
\end{proof}

We will next give an example that  the weak Harnack inequality need not hold if 
$\alpha < (\frac{n}s+1)(q-p)$.
We consider the one-dimensional case and focus on the role of $s$. 
In the double phase (or $(p,q)$-growth) case Fonseca, Mal\'y and Mingione \cite{FonMM04} 
have given more sophisticated counter-examples for $s=p$, 
which is related to the assumption \aone{} in this case. See also \cite{BalDS_pp}. 
However, to the best of our knowledge, examples for $s\ne p$ have not been considered before.

Let $\phi\in \Phiw(\R)$ be defined by $\phi(x,0):=0$ and 
\[
\phi'(x,t) := \max\{t^{p-1}, a(x) t^{q-1}\},
\]
so that $\phi(x,t) \approx \max\{t^p, a(x) t^q\} \approx t^p + a(x) t^q$, the double phase functional. 
Let $u$ be a solution of $\big(\phi'(x,|u'|)\frac {u'}{|u'|}\big)'=0$ on the interval 
$(a,b)$. We assume that $\lim_{x \to a^+}u(x)< \lim_{x\to b^-}u(x)$, so $u$ is increasing and $\frac {u'}{|u'|}=1$. 
Then the differential equation reduces to $\phi'(x,u')\equiv c$, i.e.\ 
\[
u'(x) =
\begin{cases}
c^{\frac1{p-1}}, & \text{when } c^{-\frac{q-p}{p-1}} \ge a(x), \\
(c/a(x))^{\frac1{q-1}}, & \text{otherwise.} \\
\end{cases}
\]
We further assume that $a(x):=\max\{-x,0\}^\alpha$. Since $a$ is decreasing, 
we obtain that 
\[
u'(x) =
\begin{cases}
c^{\frac1{p-1}}, & \text{when } x \ge -x_0, \\
(c |x|^{-\alpha})^{\frac1{q-1}}, & \text{when } x < -x_0, \\
\end{cases}
\]
for $x_0 := c^{-\frac1\alpha \frac{q-p}{p-1}}$. Some solutions are illustrated in Figure~\ref{fig:solutions} for different values of $c$ with zero left boundary values at $x=-1$.

\begin{figure}[ht!]
\definecolor{cqcqcq}{rgb}{0.7529411764705882,0.7529411764705882,0.7529411764705882}
\begin{tikzpicture}[line cap=round,line join=round,>=triangle 45,x=4.0cm,y=1.0cm]
\draw [color=cqcqcq,, xstep=1.0cm,ystep=1.0cm] (-1.1,-0.1) grid (1.1,6.3);
\draw[->,color=black] (-1.1,0.) -- (1.1,0.);
\foreach \x in {-1,-0.5,0.5,1}
\draw[shift={(\x,0)},color=black] (0pt,2pt) -- (0pt,-2pt) node[below] {\footnotesize $\x$};
\draw[->,color=black] (0.,-0.5) -- (0.,6.3);
\foreach \y in {1,2,3,4,5,6}
\draw[shift={(0,\y)},color=black] (2pt,0pt) -- (-2pt,0pt) node[left] {\footnotesize $\y$};
\draw[color=black] (0pt,-10pt) node[right] {\footnotesize $0$};
\clip(-1.1,-0.4) rectangle (1.1,6.3);
\draw [line width=1.2pt,domain=-1:-(1.01)^(-18)] plot(\x,{2*1.01 - 2*1.01*sqrt((-\x))}) node{$\bullet$};
\draw [line width=1.2pt,domain=-(1.01)^(-18):1] plot(\x,{2*1.01 - 2*(1.01)^(-8)+(1.01)^(10)*((\x)+(1.01)^(-18)});
\draw [line width=1.2pt,domain=-1:-(1.1)^(-18)] plot(\x,{2*1.1 - 2*1.1*sqrt((-\x))}) node{$\scriptstyle\bullet$};
\draw [line width=1.2pt,domain=-(1.1)^(-18):1] plot(\x,{2*1.1 - 2*(1.1)^(-8)+(1.1)^(10)*((\x)+(1.1)^(-18)});
\draw [line width=1.2pt,domain=-1:-(1.2)^(-18)] plot(\x,{2*1.2 - 2*1.2*sqrt((-\x))}) node{$\scriptstyle\bullet$};
\draw [line width=1.2pt,domain=-(1.2)^(-18):1] plot(\x,{2*1.2 - 2*(1.2)^(-8)+(1.2)^(10)*((\x)+(1.2)^(-18)});
\draw [smooth,samples=50,line width=1.2pt,domain=-1:-(1.3)^(-18)] plot(\x,{2*1.3 - 2*1.3*sqrt((-\x))}) node{$\scriptstyle\bullet$};
\draw [line width=1.2pt,domain=-(1.3)^(-18):0.3] plot(\x,{2*1.3 - 2*(1.3)^(-8)+(1.3)^(10)*((\x)+(1.3)^(-18)});
\draw [smooth,samples=50,line width=1.2pt,domain=-1:-(1.4)^(-18)] plot(\x,{2*1.4 - 2*1.4*sqrt((-\x))}) node{$\scriptstyle\bullet$};
\draw [line width=1.2pt,domain=-(1.4)^(-18)-0.001:0.2] plot(\x,{2*1.4-2*(1.4)^(-8)+(1.4)^(10)*((\x)+(1.4)^(-18)}); 
\begin{scriptsize}
\end{scriptsize}
\end{tikzpicture}
\caption{Solution for $c\in \{1.01, 1.1, 1.2, 1.3, 1.4\}$ in $[-1, 1]$. The parameters are  $p=1.1$, $q=2$ and $\alpha =0.5$. The right boundary values have been partly cut away but they are in the range $[2, 32]$. The 
dots indicate $x_0$.}\label{fig:solutions}
\end{figure}
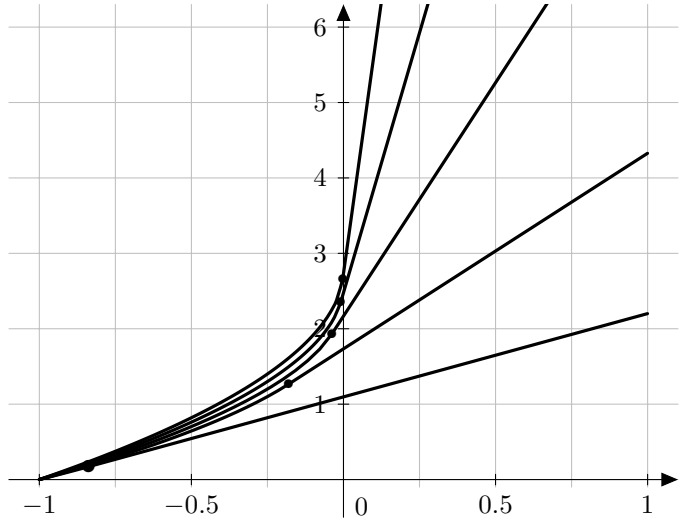

For $c>0$ and $r>x_0$, we next consider a solution with $u(-x_0-2r)=0$. When 
$\frac\alpha{q-1}\ne 1$ we have, for $\rho\in [0,r]$ and $\alpha_2:=1-\frac\alpha{1-q}$,  
\[
u(-x_0-\rho) 
= 
c^{\frac1{q-1}} \int_{-x_0-2r}^{-x_0-\rho} |x|^{-\frac \alpha{q-1}}\, dx
=
c^{\frac1{q-1}} \frac{(x_0+2r)^{\alpha_2} - (x_0+\rho)^{\alpha_2}}{\alpha_2}
\approx c^{\frac1{q-1}} r^{\alpha_2}.
\]
Furthermore, $u(-x_0+\rho) = u(-x_0) + \rho c^{\frac1{p-1}}$ since the derivative is constant on 
$(-x_0,\infty)$.  
With $c^{\frac1{p-1}-\frac1{q-1}} = c^{\frac{q-p}{(p-1)(q-1)}} = x_0^{-\frac\alpha{q-1}}$, we calculate 
\[
\frac{u(-x_0+r)}{u(-x_0)} 
=
1 + \alpha_2\frac{r c^{\frac1{p-1}-\frac1{q-1}} x_0^{-\alpha_2}}{(1+2r/x_0)^{\alpha_2} - 1}
=
1 + \alpha_2 \frac{r/x_0}{(1+2r/x_0)^{\alpha_2} - 1}.
\]
Since $\alpha_2<1$, we see that the constant in the Harnack inequality 
in $B(-x_0,r)$ blows up if $r/x_0\to \infty$. 

We next estimate the $L^s$-norm of $u$ when $s>0$. 
For $\rho\in [\frac r2,2r]$, $u(-x_0+\rho) = u(-x_0) + \rho c^{\frac1{p-1}} \approx rc^{\frac1{p-1}}$.
Since $u\ge 0$ is increasing and since only large values of the function are important when $s>0$, 
we find that 
\[
\|u\|_{L^s(B(-x_0,2r))} 
\approx c^{\frac1{p-1}} r^{1+\frac1s} 
= x_0^{1+\frac1s-\frac\alpha{q-p}} (\tfrac r{x_0})^{1+\frac1s}
\quad\text{and}\quad
\bigg(\fint_{B(-x_0,r)} u^s\, dx\bigg)^\frac1s \approx c^{\frac1{p-1}} r.
\]
Similarly, if $x\in [-x_0- r, -x_0]$, then 
by the earlier formula $u(x) \approx c^{\frac1{q-1}} r^{\alpha_2}$ and, 
since only small values of the function are important when $-s<0$, we obtain that 
\[
\bigg(\fint_{B(-x_0,r)} u^{-s}\, dx\bigg)^{-\frac1s} \approx c^{\frac1{q-1}} r^{\alpha_2}.
\]

We can therefore say in this example that 
\[
\esssup_{B(-x_0,r)} u =u(-x_0+r)\approx \bigg(\fint_{B(-x_0,r)} u^s\, dx\bigg)^\frac1s
\]
and
\[
\essinf_{B(-x_0,r)} u =u(-x_0-r) \approx \bigg(\fint_{B(-x_0,r)} u^{-s}\, dx\bigg)^{-\frac1s}
\]
for every $s>0$, with constant depending on $s$ but not on $c$ or $r$. 
It follows that a possible failure in the Harnack inequality is due to the 
passing over zero. 


As was mentioned before, the failure of the weak Harnack inequality happens if $\tfrac r{x_0}\to\infty$, 
for instance if we choose $r:= x_0 \log \frac1{x_0}$. Furthermore, 
\[
\|u\|_{L^s(B(-x_0,2r))} 
\approx 
x_0^{1+\frac1s-\frac\alpha{q-p}} (\tfrac r{x_0})^{1+\frac1s}
=
x_0^{1+\frac1s-\frac\alpha{q-p}} (\log\tfrac1{x_0})^{1+\frac1s}
\] 
remains bounded as $x_0\to 0$ if $1+\frac1s-\frac\alpha{q-p} > 0$. 
Since $n=1$, this is equivalent to $\alpha<(1+\frac ns)(p-q)$, the complement of the inequality in 
Corollary~\ref{cor:doublePhase}. This shows the sharpness of the \aones{s_*} assumption 
in Theorem~\ref{thm:weakHarnack}.

In addition, we have $\phi(x,u')\approx u' \phi'(x,u')=c u'$. Thus 
\[
\int_{B(-x_0,2r)} \phi(x,u')\, dx = c(u(-x_0+2r) - u(-x_0 -2r))
= c u(-x_0+2r) 
\approx c^{\frac p{p-1}}r
= x_0^{1-\frac{p\alpha}{q-p}} \tfrac r{x_0} .
\]
Here we see that the Harnack inequality does not hold with uniform 
constant even though the $W^{1,\phi}$-norm of $u$ remains bounded 
provided that $\frac{p\alpha}{q-p} < 1$. On the other hand, 
for the double phase functional \aone{} is equivalent to $\frac{p\alpha}{q-p}\ge 1$, 
and we showed above that the weak Harnack inequality holds in this case. 
This proves the sharpness of the \aone{} assumption in Theorem~\ref{thm:weakHarnack}.


\bigskip

\noindent\small{\textsc{A. Benyaiche}}\\
\small{Department of Mathematics, Ibn Tofail University, Kenitra, Morocco}\\
\footnotesize{\texttt{allami.benyaiche@uit.ac.ma}}\\

\noindent\small{\textsc{P. Harjulehto}}\\
\small{Department of Mathematics and Statistics,
FI-20014 University of Turku, Finland}\\
\footnotesize{\texttt{petteri.harjulehto@utu.fi}}\\

 \noindent\small{\textsc{P. Hästö}\\ 
Department of Mathematics and Statistics,
FI-20014 University of Turku, Finland}\\
\footnotesize{\texttt{peter.hasto@oulu.fi}}\\

\noindent\small{\textsc{A. Karppinen}}\\
\small{Department of Mathematics and Statistics,
FI-20014 University of Turku, Finland}\\
\footnotesize{\texttt{arttu.a.karppinen@utu.fi}}\\

\end{document}